\documentclass[12pt,reqno]{amsart}
\usepackage{amsmath, amsthm, amssymb, stmaryrd}

\advance \topmargin by -\headheight
\advance \topmargin by -\headsep
     
\evensidemargin \oddsidemargin
\newtheorem{theorem}{Theorem}[section]
\newtheorem{lemma}[theorem]{Lemma}
\newtheorem{corollary}[theorem]{Corollary}

\theoremstyle{definition}

\theoremstyle{remark}

\numberwithin{equation}{section}

\def\bfa{{\mathbf a}}
\def\bfb{{\mathbf b}}

\def\bfd{{\mathbf d}}

\def\bfi{{\mathbf i}}
\def\bfj{{\mathbf j}}
\def\bfm{{\mathbf m}}

\def\bfu{{\mathbf u}}

\def\bfx{{\mathbf x}}
\def\bfy{{\mathbf y}}
\def\bfz{{\mathbf z}}

\def\calA{{\mathcal A}}  

\def\calB{{\mathcal B}}

\def\calI{{\mathcal I}}
\def\calJ{{\mathcal J}}

\def\dbN{{\mathbb N}}  

\def\dbZ{{\mathbb Z}}

\def\alp{{\alpha}} \def\bfalp{{\boldsymbol \alpha}} \def\alptil{{\widetilde \alpha}}
\def\bet{{\beta}}  
\def\gam{{\gamma}}

\def\tet{{\theta}}

\def\sig{{\sigma}}  

\def\Ups{{\Upsilon}}

\def\ome{{\omega}} \def\Ome{{\Omega}}

\def\eps{\varepsilon}

\def\le{\leqslant} \def\ge{\geqslant}

\def\d{{\,{\rm d}}}

\begin{document}
\title[Paucity problems]{Paucity problems and some relatives of Vinogradov's mean value 
theorem}
\author[Trevor D. Wooley]{Trevor D. Wooley}
\address{Department of Mathematics, Purdue University, 150 N. University Street, West 
Lafayette, IN 47907-2067, USA}
\email{twooley@purdue.edu}
\date{}
\dedicatory{}
\begin{abstract}When $k\ge 4$ and $0\le d\le (k-2)/4$, we consider the system of 
Diophantine equations
\[
x_1^j+\ldots +x_k^j=y_1^j+\ldots +y_k^j\quad (1\le j\le k,\, j\ne k-d).
\]
We show that in this cousin of a Vinogradov system, there is a paucity of non-diagonal 
positive integral solutions. Our quantitative estimates are particularly sharp when 
$d=o(k^{1/4})$.
\end{abstract}
\maketitle

\section{Introduction} Recent progress on Vinogradov's mean value theorem has resolved 
the main conjecture in the subject. Thus, writing $J_{s,k}(X)$ for the number of integral 
solutions of the system of equations
\begin{equation}\label{1.1}
x_1^j+\ldots +x_s^j=y_1^j+\ldots +y_s^j\quad (1\le j\le k),
\end{equation}
with $1\le x_i,y_i\le X$ $(1\le i\le s)$, it is now known that whenever $\eps>0$, one has
\begin{equation}\label{1.2}
J_{s,k}(X)\ll X^{s+\eps}+X^{2s-k(k+1)/2}
\end{equation}
(see \cite{BDG2016} or \cite{Woo2016, Woo2019}). Denote by $T_s(X)$ the number of 
$s$-tuples $\bfx$ and $\bfy$ in which $1\le x_i,y_i\le X$ $(1\le i\le s)$, and 
$(x_1,\ldots ,x_s)$ is a permutation of $(y_1,\ldots ,y_s)$. Thus 
$T_s(X)=s!X^s+O(X^{s-1})$. A conjecture going beyond the main conjecture (\ref{1.2}) 
asserts that when $1\le s<\tfrac{1}{2}k(k+1)$, one should have
\begin{equation}\label{1.3}
J_{s,k}(X)=T_s(X)+o(X^s).
\end{equation}
This conclusion is essentially trivial for $1\le s\le k$, in which circumstances one has the 
definitive statement $J_{s,k}(X)=T_s(X)$. When $s\ge k+2$, meanwhile, the conclusion 
(\ref{1.3}) is at present far beyond our grasp. This leaves the special case $s=k+1$. Here, 
one has the asymptotic relation
\begin{equation}\label{1.4}
J_{k+1,k}(X)=T_{k+1}(X)+O(X^{\sqrt{4k+5}})
\end{equation}
due to the author joint with Vaughan \cite[Theorem 1]{VW1997}. An analogous conclusion 
is available when the equation of degree $k-1$ in the system (\ref{1.1}) is removed, but in 
no other close relative of Vinogradov's mean value theorem has such a conclusion been 
obtained hitherto. Our purpose in this paper is to derive estimates of strength paralleling 
(\ref{1.4}) in systems of the shape (\ref{1.1}) in which a large degree equation is 
removed.\par

In order to describe our conclusions, we must introduce some notation. When $k\ge 2$ and 
$0\le d<k$, we denote by $I_{k,d}(X)$ the number of integral solutions of the system of 
equations
\begin{equation}\label{1.5}
x_1^j+\ldots +x_k^j=y_1^j+\ldots +y_k^j\quad (1\le j\le k,\, j\ne k-d),
\end{equation}
with $1\le x_i,y_i\le X$ $(1\le i\le k)$. Also, when $k\ge 3$ and $d\ge 0$, we define the 
exponent
\begin{equation}\label{1.6}
\gam_{k,d}=\min_{2\le r\le k}\biggl( r+\frac{k}{r}+\sum_{l=1}^r\max\{d-l+1,0\}
\biggr) .
\end{equation}

\begin{theorem}\label{theorem1.1} Suppose that $k\ge 3$ and $0\le d<k/2$. Then, for each 
$\eps>0$, one has
\[
I_{k,d}(X)-T_k(X)\ll X^{\gam_{k,d}+\eps}.
\]
\end{theorem}

When $k$ is large and $d$ is small compared to $k$, the conclusion of this theorem 
provides strikingly powerful paucity estimates.

\begin{corollary}\label{corollary1.2} Suppose that $d\le \sqrt{k}$. Then
\[
I_{k,d}(X)-T_k(X)\ll X^{\sqrt{4k+1}+d(d+1)/2}.
\]
In particular, when $d=o(k^{1/4})$, one has
\[
I_{k,d}(X)=T_k(X)+O(X^{(2+o(1))\sqrt{k}}).
\]
\end{corollary}

Although for larger values of $d$ our paucity estimates become weaker, they remain 
non-trivial whenever $d<(k-2)/4$.

\begin{corollary}\label{corollary1.3} Provided that $d\ge 1$ and $k\ge 4d+3$, one has
\[
I_{k,d}(X)=k!X^k+O(X^{k-1/2}).
\]
Moreover, when $1\le d\le k/4$, one has
\[
I_{k,d}(X)-T_k(X)\ll X^{\sqrt{4k(d+1)+(d+1)^2}},
\]
so that whenever $\eta$ is small and positive, and $1\le d\le \eta^2 k$, then
\[
I_{k,d}(X)=T_k(X)+O(X^{3\eta k}).
\]
\end{corollary}

Previous work on this problem is confined to the two cases considered by Hua 
\cite[Lemmata 5.2 and 5.4]{Hua1965}. Thus, the asymptotic formula (\ref{1.4}) derived by 
the author jointly with Vaughan \cite[Theorem 1]{VW1997} is tantamount to the case 
$d=0$ of Theorem \ref{theorem1.1}. Meanwhile, it follows from 
\cite[Theorem 2]{VW1997} that
\[
I_{k,1}(X)=T_k(X)+O(X^{\gam_{k,1}-1+\eps}),
\]
and the error term here is slightly sharper than that provided by the case $d=1$ of 
Theorem \ref{theorem1.1}. The conclusion of Theorem \ref{theorem1.1} is new whenever 
$d\ge 2$. It would be interesting to derive analogues of Theorem \ref{theorem1.1} in 
which more than one equation is removed from the Vinogradov system (\ref{1.1}), or 
indeed to derive analogues in which the number of variables is increased and yet one is 
able nonetheless to confirm the paucity of non-diagonal solutions. We have more to say on 
such matters in \S5 of this paper. For now, we confine ourselves to remarking that when 
many, or even most, lower degree equations are removed, then approaches based on the 
determinant method are available. Consider, for example, natural numbers $d_1,\ldots ,d_k$ 
with $1\le d_1<d_2<\ldots <d_k$ and $d_k\ge 2s-1$. Also, denote by $M_{\bfd,s}(X)$ the 
number of integral solutions of the system of equations
\[
x_1^{d_j}+\ldots +x_s^{d_j}=y_1^{d_j}+\ldots +y_s^{d_j}\quad (1\le j\le k),
\]
with $1\le x_i,y_i\le X$ $(1\le i\le s)$. Then it follows from \cite[Theorem 5.2]{SW2010} 
that whenever $d_1\cdots d_k\ge (2s-k)^{4s-2k}$, one has 
\[
M_{\bfd,s}(X)=s!X^s+O(X^{s-1/2}).
\]

\par The proof of Theorem \ref{theorem1.1}, in common with our earlier treatment in 
\cite{VW1997} of the Vinogradov system (\ref{1.1}), is based on the application of 
multiplicative polynomial identities amongst variables in pursuit of parametrisations that 
these days would be described as being of torsorial type. The key innovation of 
\cite{VW1997} was to relate not merely two product polynomials, but instead $r\ge 2$ such 
polynomials, leading to a decomposition of the variables into $(k+1)^r$ parameters. Large 
numbers of these parameters may be determined via divisor function estimates, and 
thereby one obtains powerful bounds for the difference $J_{k+1,k}(X)-T_{k+1}(X)$. In 
the present situation, the polynomial identities are more novel, and sacrifices must be made 
in order to bring an analogous plan to fruition. Nonetheless, when $d<k/2$, the kind of 
multiplicative relations of \cite{VW1997} may still be derived in a useful form.\par

This paper is organised as follows. We being in \S2 of this paper by deriving the polynomial 
identities required for our subsequent analysis. In \S3 we refine this infrastructure so that 
appropriate multiplicative relations are obtained involving few auxiliary variables. A 
complication for us here is the problem of bounding the number of choices for these 
auxiliary variables, since they are of no advantage to us in the ensuing analysis of 
multiplicative relations. In \S4, we exploit the multiplicative relations by extracting common 
divisors between tuples of variables, following the path laid down in our earlier work 
\cite{VW1997} joint with Vaughan. This leads to the proof of Theorem \ref{theorem1.1}. 
Finally, in \S5, we discuss the corollaries to Theorem \ref{theorem1.1} and consider also 
refinements and potential generalisations of our main results.\par

Our basic parameter is $X$, a sufficiently large positive number. Whenever $\eps$ appears 
in a statement, either implicitly or explicitly, we assert that the statement holds for each 
$\eps>0$. In this paper, implicit constants in Vinogradov's notation $\ll$ and $\gg$ may 
depend on $\eps$, $k$, and $s$. We make frequent use of vector notation in the form 
$\bfx=(x_1,\ldots,x_r)$. Here, the dimension $r$ depends on the course of the argument. 
We also write $(a_1,\ldots ,a_s)$ for the greatest common divisor of the integers 
$a_1,\ldots ,a_s$. Any ambiguity between ordered $s$-tuples and corresponding 
greatest common divisors will be easily resolved by context. Finally, as usual, we write 
$e(z)$ for $e^{2\pi iz}$.\par
 
\noindent {\bf Acknowledgements:} The author's work is supported by NSF grant 
DMS-2001549 and the Focused Research Group grant DMS-1854398. 

\section{Polynomial identities} We begin by introducing the power sum polynomials
\[
s_j(\bfz)=z_1^j+\ldots +z_k^j\quad (1\le j\le k).
\]
On recalling (\ref{1.5}), we see that $I_{k,d}(X)$ counts the number of integral solutions of 
the system of equations
\begin{equation}\label{2.1}
\left.
\begin{aligned}s_j(\bfx)&=s_j(\bfy)\quad (1\le j\le k,\, j\ne k-d)\\
s_{k-d}(\bfx)&=s_{k-d}(\bfy)+h,
\end{aligned}
\right\}
\end{equation}
with $1\le \bfx,\bfy\le X$ and $|h|\le kX^{k-d}$. Our first task is to reinterpret this system 
in terms of elementary symmetric polynomials, so that our first multiplicative relations may 
be extracted.\par

The elementary symmetric polynomials $\sig_j(\bfz)\in \dbZ[z_1,\ldots ,z_k]$ may be 
defined by means of the generating function identity
\[
1+\sum_{j=1}^k\sig_j(\bfz)(-t)^j=\prod_{i=1}^k(1-tz_i).
\]
Since
\[
\sum_{i=1}^k\log (1-tz_i)=-\sum_{j=1}^\infty s_j(\bfz)\frac{t^j}{j},
\]
we deduce that
\[
1+\sum_{j=1}^k\sig_j(\bfz)(-t)^j=\exp \biggl( -\sum_{j=1}^\infty s_j(\bfz)
\frac{t^j}{j}\biggr) .
\]
When $n\ge 1$, the formula
\begin{equation}\label{2.2}
\sig_n(\bfz)=(-1)^n\sum_{\substack{m_1+2m_2+\ldots +nm_n=n\\ m_i\ge 0}}\, 
\prod_{i=1}^n \frac{(-s_i(\bfz))^{m_i}}{i^{m_i}m_i!}
\end{equation}
then follows via an application of Fa\`a di Bruno's formula. By convention, we put 
$\sig_0(\bfz)=1$. We refer the reader to \cite[equation (2.14$^\prime$)]{MacD1979} for a 
self-contained account of the relation (\ref{2.2}).\par

Suppose now that $0\le d<k/2$, and that the integers $\bfx,\bfy,h$ satisfy (\ref{2.1}). 
When $1\le n<k-d$, it follows from (\ref{2.2}) that
\begin{equation}\label{2.3}
\sig_n(\bfx)=(-1)^n\sum_{\substack{m_1+2m_2+\ldots +nm_n=n\\ m_i\ge 0}}\, 
\prod_{i=1}^n \frac{(-s_i(\bfy))^{m_i}}{i^{m_i}m_i!}=\sig_n(\bfy).
\end{equation}
When $k-d\le n\le k$, on the other hand, we instead obtain the relation
\[
\sig_n(\bfx)=(-1)^n\sum_{\substack{m_1+2m_2+\ldots +nm_n=n\\ m_i\ge 0}}\, 
\frac{(-s_{k-d}(\bfy)-h)^{m_{k-d}}}{(k-d)^{m_{k-d}}m_{k-d}!}
\prod_{\substack{1\le i\le n\\ i\ne k-d}}\frac{(-s_i(\bfy))^{m_i}}{i^{m_i}m_i!} .
\]
Since $d<k/2$, the summation condition on $\bfm$ ensures that $m_{k-d}\in \{0,1\}$. 
Thus, by isolating the term in which $m_{k-d}=1$, we see that
\begin{equation}\label{2.4}
\sig_n(\bfx)=\sig_n(\bfy)+h\psi_n(\bfy),
\end{equation}
where by (\ref{2.2}),
\begin{align*}
\psi_n(\bfy)&=\frac{(-1)^{n+1}}{k-d}\sum_{\substack{m_1+2m_2+\ldots +
(n-k+d)m_{n-k+d}=n-k+d\\ m_i\ge 0}}\prod_{i=1}^{n-k+d}
\frac{(-s_i(\bfy))^{m_i}}{i^{m_i}m_i!}\\
&=\frac{(-1)^{k-d+1}}{k-d}\sig_{n-k+d}(\bfy).
\end{align*}

\par We deduce from (\ref{2.3}) and (\ref{2.4}) that
\begin{align}
\prod_{i=1}^k(t-x_i)-\prod_{i=1}^k(t-y_i)&=(-1)^k\sum_{n=0}^k(\sig_n(\bfx)-\sig_n(\bfy))
(-t)^{k-n}\notag \\
&=(-1)^{d-1}\frac{h}{k-d}\sum_{m=0}^d\sig_m(\bfy)(-t)^{d-m}.\label{2.5}
\end{align}
Define the polynomial
\begin{equation}\label{2.6}
\tau_d(\bfy;w)=(-1)^{d-1}\sum_{m=0}^d\sig_m(\bfy)(-w)^{d-m}.
\end{equation}
Then we deduce from (\ref{2.5}) that for $1\le j\le k$, one has the relation
\begin{equation}\label{2.7}
(k-d)\prod_{i=1}^k(y_j-x_i)=\tau_d(\bfy;y_j)h.
\end{equation}

\par By comparing the relation (\ref{2.7}) with $j=s$ and $j=t$ for two distinct indices $s$ 
and $t$ satisfying $1\le s<t\le k$, it is apparent that
\begin{equation}\label{2.8}
\tau_d(\bfy;y_t)\prod_{i=1}^k(y_s-x_i)=\tau_d(\bfy;y_s)\prod_{i=1}^k(y_t-x_i).
\end{equation}
Furthermore, by applying the relations (\ref{2.3}), we see that $\sig_m(\bfy)=\sig_m(\bfx)$ 
for $1\le m\le d$, and thus it is a consequence of (\ref{2.6}) that
\begin{equation}\label{2.9}
\tau_d(\bfy;y_j)=\tau_d(\bfx;y_j)\quad (1\le j\le k).
\end{equation}
We therefore deduce from (\ref{2.8}) that for $1\le s<t\le k$, one has
\begin{equation}\label{2.10}
\tau_d(\bfx;y_t)\prod_{i=1}^k(y_s-x_i)=\tau_d(\bfx;y_s)\prod_{i=1}^k(y_t-x_i).
\end{equation}
These are the multiplicative relations that provide the foundation for our analysis. One 
additional detail shall detain us temporarily, however, for to be useful we must ensure that 
all of the factors on left and right hand sides of (\ref{2.8}) and (\ref{2.10}) are non-zero.

\par Suppose temporarily that there are indices $l$ and $m$ with $1\le l,m\le k$ for which 
$x_l=y_m$. By relabelling variables, if necessary, we may suppose that $l=m=k$, and then 
it follows from (\ref{2.1}) that
\[
x_1^j+\ldots +x_{k-1}^j=y_1^j+\ldots +y_{k-1}^j\quad (1\le j\le k,\, j\ne k-d).
\]
There are $k-1$ equations here in $k-1$ pairs of variables $x_i,y_i$, and thus it follows 
from \cite{Ste1971} that $(x_1,\ldots ,x_{k-1})$ is a permutation of $(y_1,\ldots ,y_{k-1})$. 
We may therefore conclude that in the situation contemplated at the beginning of this 
paragraph, the solution $\bfx,\bfy$ of (\ref{2.1}) is counted by $T_k(X)$, with 
$(x_1,\ldots ,x_k)$ a permutation of $(y_1,\ldots ,y_k)$. In particular, in any solution 
$\bfx,\bfy$ of (\ref{2.1}) counted by $I_{k,d}(X)-T_k(X)$, it follows that $x_l=y_m$ for no 
indices $l$ and $m$ satisfying $1\le l,m\le k$. In view of (\ref{2.7}) and (\ref{2.9}), such 
solutions also satisfy the conditions
\begin{equation}\label{2.11}
h\ne 0\quad \text{and}\quad \tau_d(\bfy;y_j)=\tau_d(\bfx;y_j)\ne 0\quad (1\le j\le k).
\end{equation}

\par We summarise the deliberations of this section in the form of a lemma.

\begin{lemma}\label{lemma2.1} Suppose that $\bfx,\bfy$ is a solution of the Diophantine 
system (\ref{2.1}) counted by $I_{k,d}(X)-T_k(X)$. Then the relations (\ref{2.8}),
(\ref{2.10}) and (\ref{2.11}) hold.
\end{lemma} 

\section{Reduction to efficient multiplicative relations} We seek to estimate the number 
$I_{k,d}(X)-T_k(X)$ of solutions of the system (\ref{2.1}), with $1\le \bfx,\bfy\le X$ and 
$|h|\le kX^{k-d}$, for which $(x_1,\ldots ,x_k)$ is not a permutation of $(y_1,\ldots ,y_k)$. 
We divide these solutions into two types according to a parameter $r$ with $1<r\le k$. Let 
$V_{1,r}(X)$ denote the number of such solutions in which there are fewer than $r$ distinct 
values amongst $x_1,\ldots ,x_k$, and likewise fewer than $r$ distinct values amongst 
$y_1,\ldots ,y_k$. Also, let $V_{2,r}(X)$ denote the corresponding number of solutions in 
which there are either at least $r$ distinct values amongst $x_1,\ldots ,x_k$, or at least $r$ 
distinct values amongst $y_1,\ldots ,y_k$. Then one has
\begin{equation}\label{3.1}
I_{k,d}(X)-T_k(X)=V_{1,r}(X)+V_{2,r}(X).
\end{equation}

\par The solutions counted by $V_{1,r}(X)$ are easily handled via an expedient argument of 
circle method flavour.

\begin{lemma}\label{lemma3.1} One has $V_{1,r}(X)\ll X^{r-1}$.
\end{lemma}

\begin{proof} It is convenient to introduce the exponential sum
\[
f(\bfalp)=\sum_{1\le x\le X}e\biggl( \sum_{\substack{1\le j\le k\\ j\ne k-d}}\alp_jx^j
\biggr) .
\]
In a typical solution $\bfx,\bfy$ of (\ref{2.1}) counted by $V_{1,r}(X)$, we may relabel 
indices in such a manner that $x_j\in \{x_1,\ldots ,x_{r-1}\}$ for $1\le j\le k$, and 
likewise $y_j\in \{y_1,\ldots ,y_{r-1}\}$ for $1\le j\le k$. On absorbing combinatorial factors 
into the constant implicit in the notation of Vinogradov, therefore, we discern via 
orthogonality that there are integers $a_i,b_i$ $(1\le i\le r-1)$, with $1\le a_i,b_i\le k$, for 
which one has
\[
V_{1,r}(X)\ll \int_{[0,1)^{k-1}}\Biggl( \prod_{i=1}^{r-1}f(a_i\bfalp)f(-b_i\bfalp)\Biggr) 
\d\bfalp .
\]

An application of H\"older's inequality shows that
\[
V_{1,r}(X)\ll \prod_{i=1}^{r-1}I(a_i)^{1/(2r-2)}I(b_i)^{1/(2r-2)},
\]
where we write
\[
I(c)=\int_{[0,1)^{k-1}}|f(c\bfalp)|^{2r-2}\d\bfalp .
\]
Thus, by making a change of variables, we discern that
\[
V_{1,r}(X)\ll \int_{[0,1)^{k-1}}|f(\bfalp)|^{2r-2}\d\bfalp .
\]
By orthogonality, the latter mean value counts the integral solutions of the system
\[
x_1^j+\ldots +x_{r-1}^j=y_1^j+\ldots +y_{r-1}^j\quad (1\le j\le k,\, j\ne k-d),
\]
with $1\le \bfx,\bfy\le X$. Since the number of equations here is $k-1$, and the number of 
pairs of variables is $r-1\le k-1$, it follows from \cite{Ste1971} that $(x_1,\ldots ,x_{r-1})$ 
is a permutation of $(y_1,\ldots ,y_{r-1})$, and hence we deduce that
\[
V_{1,r}(X)\ll T_{r-1}(X)\sim (r-1)!X^{r-1}.
\]
This establishes the upper bound claimed in the statement of the lemma.
\end{proof}

We next consider the solutions $\bfx,\bfy,h$ of the system (\ref{2.1}) counted by 
$V_{2,r}(X)$. Here, by taking advantage of the symmetry between $\bfx$ and $\bfy$, and 
if necessary relabelling indices, we may suppose that $y_1,\ldots ,y_r$ are distinct. Suppose 
temporarily that the integers $y_t$ and $x_i-y_t$ have been determined for $1\le i\le k$ and 
$1\le t\le r$. It follows that $y_t$ and $x_i$ are determined for $1\le i\le k$ and 
$1\le t\le r$, and hence also that the coefficients $\sig_m(\bfx)$ of the polynomial 
$\tau_d(\bfx;w)$ are fixed for $0\le m\le d$. The integers $y_s$ for $r<s\le k$ 
may consequently be determined from the polynomial equations (\ref{2.10}) with $t=1$. 
Here, it is useful to observe that with $y_1$ and $x_1,\ldots ,x_k$ already fixed, and all the 
factors on the left and right hand side of (\ref{2.10}) non-zero, the equation (\ref{2.10}) 
becomes a polynomial in the single variable $y_s$. On the left hand side one has a 
polynomial of degree $k$, whilst on the right hand side the polynomial has degree 
$d=\text{deg}_y(\tau_d(\bfx;y))<k$. Thus $y_s$ is determined by a polynomial of degree 
$k$ to which there are at most $k$ solutions. Given fixed choices for $y_t$ and $x_i-y_t$ for 
$1\le i\le k$ and $1\le t\le r$, therefore, there are $O(1)$ possible choices for 
$y_{r+1},\ldots ,y_k$.\par

Let $M_r(X;\bfy)$ denote the number of integral solutions $\bfx$ of the system of equations 
(\ref{2.10}) $(1\le s<t\le r)$, satisfying $1\le \bfx\le X$, wherein $\bfy=(y_1,\ldots ,y_r)$ is 
fixed with $1\le \bfy\le X$ and satisfies (\ref{2.11}). Then it follows from the above 
discussion in combination with Lemma \ref{lemma2.1} that
\begin{equation}\label{3.2}
V_{2,r}(X)\ll X^r\max_{\bfy}M_r(X;\bfy),
\end{equation}
in which the maximum is taken over distinct $y_1,\ldots ,y_r$ with $1\le \bfy\le X$.\par

Consider fixed values of $y_1,\ldots ,y_r$ with $1\le y_i\le X$ $(1\le i\le r)$. We write 
$N_r(X;\bfy)$ for the number of $r$-tuples
\begin{equation}\label{3.3}
(\tau_d(y_1,\ldots ,y_k;y_1),\ldots ,\tau_d(y_1,\ldots ,y_k;y_r)),
\end{equation}
with $1\le y_j\le X$ $(r<j\le k)$. It is apparent from (\ref{2.6}) and (\ref{2.11}) that in 
each such $r$-tuple, one has
\begin{equation}\label{3.4}
1\le |\tau_d(\bfy;y_j)|\ll X^d,
\end{equation}
and thus a trivial estimate yields the bound
\begin{equation}\label{3.5}
N_r(X;\bfy)\ll X^{rd}.
\end{equation}
On the other hand, we may consider the number of $d$-tuples
\[
(\sig_1(y_1,\ldots ,y_k),\ldots ,\sig_d(y_1,\ldots ,y_k)),
\]
with $1\le y_j\le X$ $(1\le j\le k)$. Since $|\sig_m(\bfy)|\ll X^m$ $(1\le m\le d)$, the 
number of such $d$-tuples is plainly $O(X^{d(d+1)/2})$. Recall that $\sig_0(\bfy)=1$. Then 
for each fixed choice of this $d$-tuple, and for each fixed index $j$, it follows from 
(\ref{2.6}) that the value of $\tau_d(y_1,\ldots,y_k;y_j)$ is determined. We therefore infer 
that
\begin{equation}\label{3.6}
N_r(X;\bfy)\ll X^{d(d+1)/2}.
\end{equation}
These simple estimates are already sufficient for many purposes. However, by working 
harder, one may obtain an estimate that is oftentimes superior to both (\ref{3.5}) and 
(\ref{3.6}). This we establish in Lemma \ref{lemma3.3} below. For the time being we 
choose not to interrupt our main narrative, and instead explain how bounds for 
$N_r(X;\bfy)$ may be applied to estimate $V_{2,r}(X)$.\par

When $1\le j\le r$, we substitute
\begin{equation}\label{3.7}
u_{0j}=\tau_d(\bfx;y_j)^{-1}\prod_{i=1}^r\tau_d(\bfx;y_i).
\end{equation}
Observe that there are at most $N_r(X;\bfy)$ distinct values for the integral $r$-tuple 
$(u_{01},\ldots ,u_{0r})$. Moreover, in any such $r$-tuple it follows from (\ref{3.4}) that 
$1\le |u_{0j}|\ll X^{d(r-1)}$. There is consequently a positive integer $C=C(k)$ with the 
property that, in any solution $\bfx,\bfy$ counted by $M_r(X;\bfy)$, one has 
$1\le |u_{0j}|\le CX^{d(r-1)}$.\par

Next we substitute
\[
u_{ij}=x_i-y_j\quad (1\le i\le k,\, 1\le j\le r).
\]
Then from (\ref{2.10}) we see that $M_r(X;\bfy)$ is bounded above by the number of 
integral solutions of the system
\begin{equation}\label{3.8}
\prod_{i_1=0}^ku_{i_11}=\prod_{i_2=0}^ku_{i_22}=\ldots =\prod_{i_r=0}^ku_{i_rr},
\end{equation}
with
\begin{equation}\label{3.9}
y_1+u_{i1}=y_2+u_{i2}=\ldots =y_r+u_{ir}\quad (1\le i\le k),
\end{equation}
\begin{equation}\label{3.10}
1\le |u_{ij}|\le X\quad (1\le i\le k,\, 1\le j\le r),
\end{equation}
and with $u_{0j}$ given by (\ref{3.7}) for $1\le j\le r$. Denote by $W(X;\bfy,\bfu_0)$ the 
number of integral solutions of the system (\ref{3.8}) subject to (\ref{3.9}) and 
(\ref{3.10}). Then on recalling (\ref{3.2}), we may summarise our deliberations thus far 
concerning $V_{2,r}(X)$ as follows.

\begin{lemma}\label{lemma3.2} One has
\[
V_{2,r}(X)\ll X^r\max_\bfy \left( N_r(X;\bfy)\max_{\bfu_0}W(X;\bfy,\bfu_0)\right) ,
\]
where the maximum with respect to $\bfy=(y_1,\ldots ,y_r)$ is taken over $y_1,\ldots ,y_r$ 
distinct with $1\le y_j\le X$ $(1\le j\le r)$, and the maximum over $r$-tuples 
$\bfu_0=(u_{01},\ldots ,u_{0r})$ is taken over
\[
1\le |u_{0j}|\le CX^{d(r-1)}\quad (1\le j\le r).
\]
\end{lemma}

Before fulfilling our commitment to establish an estimate for $N_r(X;\bfy)$ sharper than the 
pedestrian bounds already obtained, we introduce the exponent
\begin{equation}\label{3.11}
\tet_{d,r}=\sum_{l=1}^r\max \{ d-l+1,0\}.
\end{equation}

\begin{lemma}\label{lemma3.3} Let $d$ and $r$ be non-negative integers and let $C\ge 1$ 
be fixed. Also, let
\[
\calA_d=\{ (a_0,a_1,\ldots ,a_d)\in \dbZ^{d+1}:\text{$|a_l|\le CX^{d-l}$ $(0\le l\le d)$}\} .
\] 
Finally, when $\bfa\in \calA_d$, define
\[
f_\bfa(t)=a_0+a_1t+\ldots +a_dt^d.
\]
Suppose that $y_1,\ldots ,y_r$ are fixed integers with $1\le y_i\le X$ $(1\le i\le r)$. Then one 
has
\[
\text{card}\{ f_\bfa(y_i):\text{$\bfa\in \calA_d$ and $1\le i\le r$}\}\ll X^{\tet_{d,r}}.
\]
\end{lemma}

\begin{proof} We proceed by induction on $d$. Note first that when $d=0$, the polynomials 
$f_\bfa(t)$ are necessarily constant with $|a_0|\le C$, and thus
\[
\text{card}\{ f_\bfa(y_i): \text{$\bfa\in \calA_0$ and $1\le i\le r$}\}\le (2C+1)^r\ll 1.
\]
Since $\tet_{0,r}=0$, the conclusion of the lemma follows for $d=0$. Observe also that 
when $r=0$ the conclusion of the lemma is trivial, for then one has $\tet_{d,0}=0$ and the 
set of values in question is empty.\par

Having established the base of the induction, we proceed under the assumption that the 
conclusion of the lemma holds whenever $d<D$, for some integer $D$ with $D\ge 1$. In 
view of the discussion of the previous paragraph, we may now restrict attention to the 
situation with $d=D\ge 1$ and $r\ge 1$. Since $1\le y_r\le X$ and $y_r$ is fixed, we see 
that whenever $\bfa\in \calA_D$ one has
\begin{equation}\label{3.12}
|f_\bfa(y_r)|\le |a_0|+|a_1|y_r+\ldots +|a_D|y_r^D\le (D+1)CX^D.
\end{equation}
Put
\begin{equation}\label{3.13}
g_\bfa(y_r,t)=\frac{f_\bfa(y_r)-f_\bfa(t)}{y_r-t},
\end{equation}
so that
\[
g_\bfa(y_r,t)=\sum_{l=1}^Da_l(t^{l-1}+t^{l-2}y_r+\ldots +y_r^{l-1}).
\]
Then one sees that whenever $\bfa\in \calA_D$, one may write
\begin{equation}\label{3.14}
g_\bfa(y_r,t)=F_\bfb(t),
\end{equation}
where
\[
F_\bfb(t)=b_0+b_1t+\ldots +b_{D-1}t^{D-1},
\]
and, for $0\le l\le D-1$, one has
\[
|b_l|\le |a_{l+1}|+|a_{l+2}|y_r+\ldots +|a_D|y_r^{D-l-1}\le CDX^{D-l-1}.
\] 
Put
\[
\calB_{D-1}=\{(b_0,b_1,\ldots ,b_{D-1})\in \dbZ^D:\text{$|b_l|\le CDX^{D-1-l}$ 
$(0\le l\le D-1)$}\}.
\]
Then the inductive hypothesis for $d=D-1$ implies that
\begin{equation}\label{3.15}
\text{card}\{ F_\bfb(y_i):\text{$\bfb\in \calB_{D-1}$ and $1\le i\le r-1$}\}
\ll X^{\tet_{D-1,r-1}}.
\end{equation}

\par On recalling (\ref{3.13}) and (\ref{3.14}), we see that
\[
f_\bfa(y_i)=f_\bfa(y_r)-(y_r-y_i)F_\bfb(y_i)\quad (1\le i\le r-1).
\]
The values of $y_i$ $(1\le i\le r-1)$ are fixed, and by (\ref{3.15}) there are 
$O(X^{\tet_{D-1,r-1}})$ possible choices for $F_\bfb(y_i)$ $(1\le i\le r-1)$. Then for each 
fixed choice of $f_\bfa(y_r)$, there are $O(X^{\tet_{D-1,r-1}})$ choices available for 
$f_\bfa(y_i)$ $(1\le i\le r-1)$. We therefore deduce from (\ref{3.12}) that
\[
\text{card}\{ f_\bfa(y_i):\text{$\bfa\in \calA_D$ and $1\le i\le r$}\}\ll 
X^D\cdot X^{\tet_{D-1,r-1}}.
\]
Since, from (\ref{3.11}), one has
\begin{align*}
\tet_{D-1,r-1}+D&=D+\sum_{l=1}^{r-1}\max \{ (D-1)-l+1,0\}\\
&=\sum_{l=1}^r\max \{ D-l+1,0\}=\tet_{D,r},
\end{align*}
we find that
\[
\text{card}\{ f_\bfa(y_i):\text{$\bfa\in \calA_D$ and $1\le i\le r$}\}\ll X^{\tet_{D,r}}.
\]
The inductive hypothesis therefore follows for $d=D$ and all values of $r$. The conclusion 
of the lemma consequently follows by induction.
\end{proof}

On recalling (\ref{2.6}), a brief perusal of (\ref{3.3}) and the definition of $N_r(X;\bfy)$ 
leads from Lemma \ref{lemma3.3} to the estimate $N_r(X;\bfy)\ll X^{\tet_{d,r}}$. We may 
therefore conclude this section with the following upper bound for $I_{k,d}(X)-T_k(X)$.

\begin{lemma}\label{lemma3.4} One has
\[
I_{k,d}(X)-T_k(X)\ll X^{r-1}+X^{r+\tet_{d,r}}\max_{\bfy ,\bfu_0}W(X;\bfy,\bfu_0),
\]
where the maximum is taken over distinct $y_1,\ldots ,y_r$ with $1\le y_j\le X$ and over
$1\le |u_{0j}|\le CX^{d(r-1)}$ $(1\le j\le r)$.
\end{lemma}

\begin{proof} It follows from Lemma \ref{lemma3.2} together with the bound for 
$N_r(X;\bfy)$ just obtained that
\[
V_{2,r}(X)\ll X^{r+\tet_{d,r}}\max_{\bfy ,\bfu_0}W(X;\bfy,\bfu_0).
\]
The conclusion of the lemma is obtained by substituting this estimate together with that 
supplied by Lemma \ref{lemma3.1} into (\ref{3.1}).
\end{proof}

\section{Exploiting multiplicative relations} Our goal in this section is to estimate the 
quantity $W(X;\bfy,\bfu_0)$ that counts solutions of the multiplicative equations (\ref{3.8}) 
equipped with their ancillary conditions (\ref{3.9}) and (\ref{3.10}). For this purpose, we 
follow closely the trail first adopted in our work with Vaughan \cite[\S2]{VW1997}.

\begin{lemma}\label{lemma4.1} Suppose that $y_1,\ldots ,y_r$ are distinct integers with 
$1\le \bfy\le X$, and that $u_{0j}$ $(1\le j\le r)$ are integers with 
$1\le |u_{0j}|\le CX^{d(r-1)}$. Then one has $W(X;\bfy,\bfu_0)\ll X^{k/r+\eps}$.
\end{lemma}

\begin{proof} We begin with a notational device from \cite[\S2]{VW1997}. Let $\calI$ 
denote the set of indices $\bfi=(i_1,\ldots ,i_r)$ with $0\le i_m\le k$ $(1\le m\le r)$. Define 
the map $\varphi:\calI \rightarrow [0,(k+1)^r)\cap \dbZ$ by putting
\[
\varphi(\bfi)=\sum_{m=1}^ri_m(k+1)^{m-1}.
\]
The map $\varphi$ is bijective, and we may define the {\it successor} $\bfi+1$ of the index 
$\bfi$ by means of the relation
\[
\bfi+1=\varphi^{-1}(\varphi(\bfi)+1).
\]
We then define $\bfi+h$ inductively via the formula $\bfi+(h+1)=(\bfi+h)+1$. Finally, when 
$\bfi\in \calI$, we write $\calJ(\bfi)$ for the set of indices $\bfj\in \calI$ having the 
property that, for some $h\in \dbN$, one has $\bfj+h=\bfi$. Thus, the set $\calJ(\bfi)$ 
is the set of all precursors of $\bfi$, in the natural sense.\par

Equipped with this notation, we now explain how systematically to extract common factors 
between the variables in the system of equations (\ref{3.8}). Put
\[
\alp_{\bf 0}=(u_{01},u_{02},\ldots ,u_{0r}),
\]
noting that by hypothesis, this integer is fixed. Suppose at stage $\bfi$ that $\alp_\bfj$ has 
been defined for all $\bfj\in \calJ(\bfi)$. We then define
\[
\alp_\bfi=\left( \frac{u_{i_11}}{\bet^{(1)}_\bfi},\frac{u_{i_22}}{\bet^{(2)}_\bfi},\ldots 
,\frac{u_{i_rr}}{\bet^{(r)}_\bfi}\right) ,
\]
in which we write
\[
\bet^{(m)}_\bfi=\prod_{\substack{\bfj\in \calJ(\bfi)\\ j_m=i_m}}\alp_\bfj .
\]
As is usual, the empty product is interpreted to be $1$. As a means of preserving intuition 
concerning the numerous variables generated in this way, we write
\[
\alptil_{lm}^\pm=\pm\prod_{\substack{\bfj\in \calI\\ j_m=l}}\alp_\bfj\quad 
(0\le l\le k,\, 1\le m\le r).
\]
Then, much as in \cite[\S2]{VW1997}, it follows that when $0\le l\le k$ and $1\le m\le r$, 
for some choice of the sign $\pm$, one has $u_{lm}=\alptil_{lm}^\pm$. Note here that the 
ambiguity in the sign of $u_{lm}$ relative to $|\alptil_{lm}^\pm|$ is a feature overlooked in 
the treatment of \cite{VW1997}, though the ensuing argument requires no significant 
modification to be brought to play in order that the same conclusion be obtained. At worst, 
an additional factor $2^{r(k+1)}$ would need to be absorbed into the constants implicit in 
Vinogradov's notation.\par

With this notation in hand, it follows from its definition that $W(X;\bfy,\bfu_0)$ is bounded 
above by the number $\Ome_r(X;\bfy,\bfu_0)$ of solutions of the system
\begin{equation}\label{4.1}
y_1+\alptil_{i1}^\pm=y_2+\alptil_{i2}^\pm=\ldots =y_r+\alptil_{ir}^\pm\quad (1\le i\le k),
\end{equation}
with
\begin{equation}\label{4.2}
1\le |\alptil_{ij}^\pm|\le X\quad (1\le i\le k,\, 1\le j\le r).
\end{equation}
Notice here that $\alptil^\pm_{0m}=u_{0m}$. Thus, it follows from a divisor function 
estimate that when the integers $u_{0m}$ are fixed with
\[
1\le |u_{0m}|\le CX^{d(r-1)}\quad (1\le m\le r),
\]
then there are $O(X^\eps)$ possible choices for the variables $\alp_\bfi$ having the 
property that $i_m=0$ for some index $m$ with $1\le m\le r$.\par

Having carefully prepared the notational infrastructure to make comparison with 
\cite[\S\S2 and 3]{VW1997} transparent, we may now follow the argument of the latter 
mutatis mutandis. When $1\le p\le r$, we write
\begin{equation}\label{4.3}
B_p=\prod_\bfi \alp_\bfi,
\end{equation}
where the product is taken over all $\bfi\in \calI$ with $i_l>i_p$ $(l\ne p)$, and $i_l>0$ 
$(1\le l\le r)$. Thus, in view of (\ref{4.2}), one has
\[
\prod_{p=1}^rB_p\le \prod_{\substack{\bfi\in \calI\\ i_l>0\, (1\le l\le r)}}
\alp_\bfi \le \prod_{i=1}^k|\alptil^\pm_{i1}|\le X^k,
\]
and so in any solution $\bfalp^\pm$ of (\ref{4.1}) counted by $\Ome_r(X;\bfy,\bfu_0)$, 
there exists an index $p$ with $1\le p\le r$ such that
\begin{equation}\label{4.4}
1\le B_p\le X^{k/r}.
\end{equation}
By relabelling variables, we consequently deduce that
\[
\Ome_r(X;\bfy,\bfu_0)\ll \Ups_r(X;\bfy,\bfu_0),
\]
where $\Ups_r(X;\bfy,\bfu_0)$ denotes the number of integral solutions of the system
\begin{equation}\label{4.5}
\alptil_{i1}^\pm-\alptil_{ij}^\pm =L_j\quad (1\le i\le k,\, 2\le j\le r),
\end{equation}
with $L_j=y_j-y_1$ $(2\le j\le r)$, and with the integral tuples $\alp_\bfi$ satisfying 
(\ref{4.2}) together with the inequality
\begin{equation}\label{4.6}
1\le B_1\le X^{k/r}.
\end{equation}
We emphasise here that, when $y_1,\ldots ,y_r$ are distinct, then $L_j\ne 0$ $(2\le j\le r)$.

\par We now proceed under the assumption that $y_1,\ldots ,y_r$ are fixed and distinct, 
whence the integers $L_j$ $(2\le j\le r)$ are fixed and non-zero. It follows just as in the 
final paragraphs of \cite[\S2]{VW1997} that, when the variables $\alp_\bfi$, with 
$\bfi\in \calI$ satisfying $i_l>i_1$ $(2\le l\le r)$, are fixed, then there are $O(X^\eps)$ 
possible choices for the tuples $\alp_\bfi$ satisfying (\ref{4.2}) and (\ref{4.5}). Here we 
make use of the fact that the variables $\alp_\bfi$, in which $i_m=0$ for some index $m$ 
with $1\le m\le r$, may be considered fixed with the potential loss of a factor $O(X^\eps)$ 
in the resulting estimates. By making use of standard estimates for the divisor function, 
however, we find from (\ref{4.6}) and the definition (\ref{4.3}) that there are 
$O(X^{k/r+\eps})$ possible choices for the variables $\alp_\bfi$ with $\bfi\in \calI$ 
satisfying $i_l>i_1$ $(2\le l\le r)$. We therefore infer that 
$\Ups_r(X;\bfy,\bfu_0)\ll X^{k/r+\eps}$, whence $\Ome_r(X;\bfy,\bfu_0)\ll X^{k/r+\eps}$, 
and finally $W(X;\bfy,\bfu_0)\ll X^{k/r+\eps}$. This completes the proof of the lemma.
\end{proof}

The proof of Theorem \ref{theorem1.1} is now at hand. By applying Lemma \ref{lemma3.4} 
in combination with Lemma \ref{lemma4.1}, we obtain the upper bound
\[
I_{k,d}(X)-T_k(X)\ll X^{r+\tet_{d,r}}\cdot X^{k/r+\eps}.
\]
By minimising the right hand side over $2\le r\le k$, a comparison of (\ref{1.6}) and 
(\ref{3.11}) now confirms that this estimate delivers the one claimed in the statement of 
Theorem \ref{theorem1.1}. 

\section{Corollaries and refinements} We complete our discussion of incomplete Vinogradov 
systems by first deriving the corollaries to Theorem \ref{theorem1.1} presented in the 
introduction, and then considering refinements to the main strategy.

\begin{proof}[The proof of Corollary \ref{corollary1.2}] Suppose that $d\le \sqrt{k}$ and 
take $r$ to be the integer closest to $\sqrt{k}$. Thus $d\le r$ and we find from (\ref{1.6}) 
that
\[
\gam_{k,d}\le r+k/r+d(d+1)/2<\sqrt{4k+1}+d(d+1)/2.
\]
An application of Theorem \ref{theorem1.1} therefore leads us to the asymptotic formula 
\[
I_{k,d}(X)=T_k(X)+O(X^{\sqrt{4k+1}+d(d+1)/2}),
\]
confirming the first claim of the corollary. In particular, when $d=o(k^{1/4})$, we discern 
that 
\[
\sqrt{4k+1}+d(d+1)/2\le \sqrt{4k+1}+o(k^{1/2})=(2+o(1))\sqrt{k},
\]
and so the final claim of the corollary follows.
\end{proof}

\begin{proof}[The proof of Corollary \ref{corollary1.3}] Suppose that $d\ge 1$ and 
$k\ge 4d+3$. In this situation, by reference to (\ref{1.6}) with $r=2$, we find that
\[
\gam_{k,d}\le 2+\tfrac{1}{2}k+2d-1=\tfrac{1}{2}(k+4d+2)\le k-\tfrac{1}{2}.
\]
Consequently, it follows from Theorem \ref{theorem1.1} that 
$I_{k,d}(X)-T_k(X)\ll X^{k-1/2}$, so that the first claim of the corollary follows.\par

Next by considering (\ref{1.6}) with $r$ taken to be the integer closest to $\sqrt{k/(d+1)}$, 
we find that
\[
\gam_{k,d}\le rd+(r+k/r)\le (d+1)\sqrt{4k/(d+1)+1}.
\]
In this instance, Theorem \ref{theorem1.1} supplies the asymptotic formula
\[
I_{k,d}(X)=T_k(X)+O(X^{\sqrt{4k(d+1)+(d+1)^2}}),
\]
which establishes the second claim of the corollary.\par

Finally, when $\eta$ is small and positive, and $1\le d\le \eta^2 k$, one finds that
\[
\gam_{k,d}\le \sqrt{4\eta^2k^2+\eta^4 k^2+(4+2\eta^2)k+1}<3\eta k.
\]
The final estimate of the corollary follows, and this completes the proof. 
\end{proof}

Some refinement is possible within the argument applied in the proof of Theorem 
\ref{theorem1.1} for smaller values of $k$. Thus, an argument analogous to that discussed 
in the final paragraph of \cite[\S2]{VW1997} shows that the bound $1\le B_p\le X^{k/r}$ 
of equation (\ref{4.4}) may be replaced by the corresponding bound
\[
1\le B_p\le X^{\ome(k,r)},
\]
where we write
\[
\ome(k,r)=k^{1-r}\sum_{i=1}^{k-1}i^{r-1}.
\]
In order to justify this assertion, denote by $\calI^+$ the set of indices $\bfi\in \calI$ such 
that $i_l>0$ $(1\le l\le r)$, and let $\calI^*$ denote the corresponding set of indices 
subject to the additional condition that for some index $p$ with $1\le p\le r$, one has 
$i_l>i_p$ whenever $l\ne p$. Then, just as in \cite[\S2]{VW1997}, one has 
$\text{card}(\calI^+)=k^r$ and $\text{card}(\calI^*)=r\psi_r(k)$, where
\[
\psi_r(k)=\sum_{i=1}^{k-1}i^{r-1}<k^r/r.
\]

\par In the situation of the proof of Lemma \ref{lemma4.1} in \S4, the variables $\alp_\bfi$ 
with $i_l=0$ for some index $l$ with $1\le l\le r$ are already determined via a divisor 
function estimate. By permuting and relabelling indices $i_l$, for each fixed index $l$, as 
necessary, the argument of the proof can be adapted to show that 
$W(X;\bfy,\bfu_0)\ll Y_r(X)$, where $Y_r(X)$ denotes the number of solutions $\bfalp^\pm$ 
as before, but subject to the additional condition
\[
\prod_{\bfi\in \calI^*}\alp_\bfi \le \Biggl( \prod_{\bfi\in \calI^+}\alp_\bfi \Biggr)^{
\text{card}(\calI^*)/\text{card}(\calI^+)}.
\]
Then
\[
\prod_{p=1}^rB_p\le \prod_{\bfi\in \calI^*}\alp_\bfi \le (X^k)^{r\psi_r(k)/k^r}.
\]
Consequently, in any solution $\bfalp^\pm$ of (\ref{4.1}) counted by 
$\Ome_r(X;\bfy;\bfu_0)$, there exists an index $p$ with $1\le p\le r$ such that
\[
1\le B_p\le X^{\psi_r(k)/k^{r-1}}=X^{\ome(k,r)}.
\]
By pursuing the same argument as in our earlier treatment, mutatis mutandis, we now 
derive the upper bound
\[
I_{k,d}(X)-T_k(X)\ll X^{\gam^\prime_{k,d}+\eps},
\]
where
\[
\gam^\prime_{k,d}=\min_{2\le r\le k}\biggl( r+\ome(k,r)+\sum_{l=1}^r\max 
\{ d-l+1,0\}\biggr).
\]

\par We conclude from these deliberations that Theorem \ref{theorem1.1} and the 
first conclusion of Corollary \ref{corollary1.3} may be refined as follows. 

\begin{theorem}\label{theorem5.1} Suppose that $k\ge 3$ and $0\le d<k/2$. Then, for 
each $\eps>0$, one has
\[
I_{k,d}(X)-T_k(X)\ll X^{\gam^\prime_{k,d}+\eps},
\]
where
\[
\gam^\prime_{k,d}=\min_{2\le r\le k}\biggl( r+k^{1-r}\sum_{i=1}^{k-1}i^{r-1}+
\sum_{l=1}^r\max\{ d-l+1,0\}\biggr) .
\]
In particular, provided that $d\ge 1$ and $k\ge 4d+2$, one has
\[
I_{k,d}(X)=k!X^k+O(X^{k-1/2}).
\]
\end{theorem}

\begin{proof} The proof of the first conclusion has already been outlined. As for the second, 
by taking $r=2$ we discern that
\[
\gam^\prime_{k,d}\le 2+\tfrac{1}{2}(k-1)+2d-1.
\]
Thus, provided that $k>4d+1$, one finds that $\gam^\prime_{k,d}\le k-1/2$, and hence the 
final conclusion of the theorem follows from the first.
\end{proof}

Energetic readers will find a smorgasbord of problems to investigate allied to those 
examined in this paper. We mention three in order to encourage work on these topics.\par

We begin by noting that the conclusions of Theorem \ref{theorem1.1} establish the paucity 
of non-diagonal solutions in the system (\ref{1.5}) when $d$ is smaller than about $k/4$. 
In principle, the methods employed remain useful when $d<k/2$. However, when $d>k/2$ 
the analogue of the identity (\ref{2.4}) that would be obtained would contain terms involving 
$h^2$, or even larger powers of $h$, and this precludes the possibility of eliminating all of 
the terms involving $h$ in any useful manner. A simple test case would be the situation with 
$d=k-1$, wherein the system (\ref{1.5}) assumes the shape
\[
x_1^j+\ldots +x_k^j=y_1^j+\ldots +y_k^j\quad (2\le j\le k).
\]
When $k=3$ an affine slicing approach has been employed in \cite{Woo1996} to resolve 
the associated paucity problem. It would be interesting to address this problem when 
$k\ge 4$.\par

The focus of this paper has been on the situation in which one slice is removed from a 
Vinogradov system. When more than one slice is removed, two or more auxiliary variables 
$h_1,h_2,\ldots $ take the place of the single variable $h$ in the identity (\ref{2.4}), and 
this seems to pose serious problems for our methods. A simple test case in this context 
would address the system of equations
\[
x_1^j+\ldots +x_k^j=y_1^j+\ldots +y_k^j\quad (j\in \{1,2,\ldots ,k-2,k+1\}),
\]
with $k\ge 3$. Here, the situation with $k=3$ has been successfully addressed by a number 
of authors (see \cite{Gre1997, SW1997} and \cite[Corollary 0.3]{Sal2007}), but little seems 
to be known for $k\ge 4$. Much more is known when the omitted slices are carefully chosen 
so that the resulting systems assume a special shape. Most obviously, one could consider 
systems of the shape
\[
x_1^{tj}+\ldots +x_k^{tj}=y_1^{tj}+\ldots +y_k^{tj}\quad (1\le j\le k-1).
\]
By specialising variables, one finds from \cite[Theorem 1]{VW1997} that the number of 
non-diagonal solutions of this system with $1\le \bfx,\bfy\le X$ is $O(X^{t\sqrt{4k+1}})$, 
and this is $o(T_k(X))$ provided only that the integer $t$ is smaller than 
$\tfrac{1}{2}\sqrt{k}-1$. Moreover, the ingenious work of Br\"udern and Robert 
\cite{BR2012} shows that when $k\ge 4$, there is a paucity of non-diagonal solutions to 
systems of the shape
\[
x_1^{2j-1}+\ldots +x_k^{2j-1}=y_1^{2j-1}+\ldots +y_k^{2j-1}\quad (1\le j\le k-1),
\]
wherein all of the even degree slices are omitted. A strategy for systems having arbitrary 
exponents can be extracted from \cite{Woo1993}, though the work there misses a paucity 
estimate by a factor $(\log X)^A$, for a suitable $A>0$.\par

We remark finally that the system of equations (\ref{1.5}) central to Theorem 
\ref{theorem1.1} has the property that there are $k-1$ equations and $k$ pairs of variables 
$x_i,y_i$. No paucity result is available when the number of pairs of variables exceeds $k$. 
The simplest challenge in this direction would be to establish that when $k\ge 3$, one has
\[
J_{k+2,k}(X)=T_{k+2}(X)+o(X^{k+2}).
\]

\bibliographystyle{amsbracket}
\providecommand{\bysame}{\leavevmode\hbox to3em{\hrulefill}\thinspace}

\end{document}